\documentclass[10pt]{amsart}

\usepackage[latin1]{inputenc}
\usepackage[english]{babel}
\usepackage{indentfirst}
\usepackage{amssymb}
\usepackage{amsthm}
\usepackage[]{graphicx}
\usepackage[all]{xy}

\newcommand{\impli}{\Rightarrow}
\newcommand{\Nat}{\mathbb{N}}
\newcommand{\cT}{\mathcal{T}}
\newcommand{\cS}{\mathcal{S}}
\def\epsilon{\varepsilon}
\newcommand{\sub}{\subseteq}

\newtheorem{theo}{Theorem}[section]
\newtheorem{lem}[theo]{Lemma}
\newtheorem{pro}[theo]{Proposition}
\newtheorem{cor}[theo]{Corollary}
\newtheorem{defi}[theo]{Definition}
\newtheorem{fact}[theo]{Fact}
\newtheorem{rem}[theo]{Remark}

\newtheorem{problem}[theo]{Problem}

\numberwithin{equation}{section}

\author{Jos\'{e} Rodr\'{i}guez}
\address{Dpto. de Ingenier\'{i}a y Tecnolog\'{i}a de Computadores,
Facultad de Inform\'{a}tica, Universidad de Murcia, 30100 Espinardo (Murcia), Spain} \email{joserr@um.es}



\thanks{Research supported by Ministerio de Econom\'{i}a y Competitividad and 
FEDER (project MTM2014-54182-P). This work was also supported by the 
research project 19275/PI/14 funded by Fundaci\'{o}n S\'{e}neca - Agencia de Ciencia y Tecnolog\'{i}a 
de la Regi\'{o}n de Murcia within the framework of PCTIRM 2011-2014.}

\title{A class of weakly compact sets in Lebesgue-Bochner spaces}

\begin{document}

\begin{abstract}
Let $X$ be a Banach space and $\mu$ a probability measure. A set $K \sub L^1(\mu,X)$ is said to be a 
$\delta\mathcal{S}$-set if it is uniformly integrable and for every $\delta>0$ there is a weakly compact set $W \sub X$ such that
$\mu(f^{-1}(W)) \geq 1-\delta$ for every $f\in K$. This is a sufficient, but in general non necessary, 
condition for relative weak compactness in $L^1(\mu,X)$. 
We say that $X$ has property ($\delta\cS_\mu$) if every relatively weakly compact subset of $L^1(\mu,X)$ is a $\delta\mathcal{S}$-set.
In this paper we study $\delta\mathcal{S}$-sets and Banach spaces having property ($\delta\cS_\mu$).
We show that testing on uniformly bounded sets is enough to check this property. 
New examples of spaces having property ($\delta\cS_\mu$) are provided.
Special attention is paid to the relationship with strongly weakly compactly generated (SWCG) spaces.
In particular, we show an example of a SWCG (in fact, separable Schur) space failing property 
($\delta\mathcal{S}_\mu$) when $\mu$ is the Lebesgue measure on $[0,1]$.
\end{abstract}

\maketitle

\section{Introduction}

Weak compactness is perhaps one of the most interesting topics in the theory of Lebesgue-Bochner spaces. 
Many efforts to understand weakly compact sets in such spaces were done prior the striking characterization found by 
\"{U}lger~\cite{ulg} and Diestel, Ruess and Schachermayer~\cite{die-alt2}. One of those attempts was based on
the natural idea of generating a weakly compact set of vector-valued functions from a weakly compact set of the Banach space in the range.
This idea was developed in \cite{bat-hie,bou-JJ,die4} leading to the notion of {\em $\delta\mathcal{S}$-set}.
Throughout this paper $X$ is a Banach space and $(\Omega,\Sigma,\mu)$ is a probability space.

\begin{defi}\label{defi:deltaS-set}
A set $K \sub L^1(\mu,X)$ is said to be a $\delta\mathcal{S}$-set if it is uniformly integrable and
for every $\delta>0$ there exists a weakly compact set $W \sub X$ such that
$\mu(f^{-1}(W))\geq 1-\delta$ for every $f\in K$.
\end{defi}

The following result was essentially proved in \cite{die4}:

\begin{theo}[Diestel]\label{theo:Diestel}
Every $\delta\mathcal{S}$-set of~$L^1(\mu,X)$ is relatively weakly compact.
\end{theo}
     
The converse of Theorem~\ref{theo:Diestel} is true if $\mu$ is purely atomic and also, for an arbitrary probability~$\mu$, whenever
$X$ is reflexive or $X=L^1(\nu)$ for any non-negative measure~$\nu$, see \cite[Theorem~4.3]{bat-hie} (cf. \cite[Theorem~14]{bou-JJ}).
A typical example of relatively weakly compact but non $\delta\cS$-set is $\{r_n x_n:n\in \Nat\} \sub L^1([0,1],X)$, where
$(r_n)$ is the sequence of Rademacher functions and $(x_n)$ is a weakly Cauchy sequence in~$X$ which is not weakly convergent
(see \cite[Proposition~4.1]{bat-hie}); this construction can be done whenever $X$ is not weakly sequentially complete (WSC for short). 
A nice example of Batt and Hiermeyer showed that the converse of Theorem~\ref{theo:Diestel} may also fail for a WSC space (see \cite[Section~3]{bat-hie}). 
  
Our motivation to study $\delta\cS$-sets comes from several attempts to solve an open problem of Schl\"{u}chtermann and Wheeler~\cite{sch-whe} asking
if the property of being strongly weakly compactly generated (SWCG for short) passes from $X$ to $L^1(\mu,X)$, see \cite{laj-rod,laj-rod-2,rod13}. To be more
precise we need a definition:

\begin{defi}\label{defi:StronglyGenerated}  
A family $\mathcal{C}$ of subsets of a Banach space~$Z$ is said to be strongly generated if there exists $C_0 \in \mathcal{C}$ such that, 
for every $C \in \mathcal{C}$ and every $\varepsilon>0$, there is $n\in \mathbb{N}$ such that $C \subseteq nC_0 + \varepsilon B_Z$.
\end{defi}

A Banach space is called SWCG if the family of its weakly compact sets is strongly generated. This class of spaces is included in that of WSC spaces 
and contains all reflexive spaces, all $L^1$-spaces of a probability measure and all separable Schur spaces. It is not difficult
to check that $X$ is SWCG if and only if the family of all $\delta\cS$-sets of $L^1(\mu,X)$ is strongly generated (see \cite[Theorem~2.7]{rod13}).
Thus, in \cite[Problem~2.10]{rod13} we raised the following question:

\begin{problem}\label{problem:JR}
Suppose that $X$ is SWCG. Is every relatively weakly compact subset of $L^1(\mu,X)$ a $\delta\cS$-set?
\end{problem}

In Section~\ref{section:examples} we study the connection between SWCG spaces and $\delta\cS$-sets.
We show that Problem~\ref{problem:JR} has negative answer in general. The example is based on
the construction of Talagrand~\cite{tal15} of a separable Schur dual Banach space $E^*$ for which
$L^1([0,1],E^*)$ fails the Dunford-Pettis property. We also revisit the aforementioned example
of Batt and Hiermayer, and provide yet another proof that their Banach space is not SWCG (see \cite[Example~2.6]{sch-whe} and
\cite[4.6]{sch-whe-2}), neither a subspace of a SWCG space (see \cite[Example~2.9]{mer-sta-2}).

In Section~\ref{section:results} we give some further results on $\delta\cS$-sets and Banach spaces having 
the following property:

\begin{defi}\label{defi:DeltaS-space}
We say that $X$ has property ($\delta\cS_\mu$) if every relatively weakly compact subset of~$L^1(\mu,X)$ is a $\delta\mathcal{S}$-set.
\end{defi}

Testing on uniformly bounded sets is enough to check this property (Theorem~\ref{theo:UniformlyBounded}). We also 
prove its stability by countable $\ell^r$-sums for any $1\leq r < \infty$  (Theorem~\ref{theo:StabilitySums}).
We finish the paper with some remarks on the concept of $\delta\cS$-set in the space $L^p(\mu,X)$ for $1<p<\infty$.

\subsection*{Terminology}
Our Banach spaces are over the real field.  
Given a Banach space~$Z$, its norm is denoted by $\|\cdot\|_Z$. 
The symbol $B_Z$ stands for the closed unit ball of~$Z$. 
The topological dual of~$Z$ is denoted by~$Z^*$. By a {\em subspace} of~$Z$ we mean a closed linear subspace.
We say that a Banach space~$Y$ is a subspace of~$Z$ if it is isomorphic to a subspace of~$Z$.
An {\em operator} is a linear continuous map between Banach spaces.

Given $A\in \Sigma$, its characteristic (or indicator) function is denoted by $\chi_A$. We write 
$\Sigma_A:=\{B\in \Sigma:B \sub A\}$ for the trace $\sigma$-algebra on~$A$. If $\mu(A)>0$, then
$\mu_A$ denotes the probability measure on~$\Sigma_A$ defined by $\mu_A(B):=\mu(B)/\mu(A)$ for all $B\in \Sigma_A$.

A function $f:\Omega\to X$ is called:
\begin{itemize}
\item {\em simple} if $f=\sum_{i=1}^n \chi_{A_i}x_i$, where $n\in \Nat$, $x_i\in X$ and $A_i\in \Sigma$;
\item {\it strongly measurable} if there is a sequence of simple functions $f_n:\Omega\to X$ 
converging to~$f$ $\mu$-a.e.;
\item {\em Bochner integrable} if it is strongly measurable and $\int_\Omega \|f(\cdot)\|_X \, d\mu<\infty$.
\end{itemize}

We denote by $L^1(\mu,X)$ the Banach space of all (equivalence classes of) Bochner integrable functions $f:\Omega\to X$, equipped with the norm 
$$
	\|f\|_{L^1(\mu,X)}=\int_\Omega \|f(\cdot)\|_X \, d\mu.
$$
The monographs \cite{cem-men,die-uhl-J,lin-J} are excellent sources of information on the {\em Lebesgue-Bochner space} $L^1(\mu,X)$. 
As usual, we denote in the same way a single Bochner integrable function and its equivalence class in~$L^1(\mu,X)$.
Any strongly measurable function coincides $\mu$-a.e. with a $\Sigma$-${\rm Borel}(X)$ measurable function. In particular,
each $f\in L^1(\mu,X)$ induces a probability measure $\mu_f$ on ${\rm Borel}(X)$ via the formula 
$$
	\mu_f(C):=\mu(f^{-1}(C)).
$$

Recall that a set $H\sub L^1(\mu,X)$ is said to be {\em uniformly integrable} if it is bounded
and for every $\epsilon>0$ there is $\delta>0$ such that $\sup_{f\in H}\int_A \|f(\cdot)\|_X \, d\mu\leq \epsilon$
for every $A\in \Sigma$ with $\mu(A)\leq\delta$. It is well-known
that every relatively weakly compact subset of $L^1(\mu,X)$ is uniformly integrable (see e.g. \cite[p.~104, Theorem~4]{die-uhl-J}).

We denote by $\lambda$ the Lebesgue measure on the $\sigma$-algebra~$\mathcal{L}$
of all Lebesgue measurable subsets of~$[0,1]$. As usual, we write $L^1([0,1],X)$ instead of $L^1(\lambda,X)$.

\section{SWCG spaces and $\delta\cS$-sets}\label{section:examples}

The results on the examples of Talagrand and Batt-Hiermeyer are included at the end of this section (see Subsections~\ref{subsection:Talagrand}
and~\ref{subsection:BH}). To deal with them we need some previous work on $\delta\cS$-sets, property~($\delta\cS_\mu$) and SWCG spaces.

Lemma~\ref{lem:MeasureAlgebra} below shows that property~($\delta\cS_{\mu}$)
depends only on the measure algebra of~$(\Omega,\Sigma,\mu)$. For complete information on measure algebras we refer to \cite{fre14}. 
In the proof of Lemma~\ref{lem:MeasureAlgebra} we will use the following standard result.

\begin{fact}\label{fact:approx}
Let $f:\Omega \to X$ be a strongly measurable function and $A\in \Sigma$. Then there is a sequence of simple functions $f_n:\Omega \to X$ such that:
\begin{enumerate}
\item[(i)] $f_n \to f$ $\mu$-a.e.;
\item[(ii)] $\|f_n(\omega)\|_X \leq 1/n+\|f(\omega)\|_X$ for every $\omega\in \Omega$ and every $n\in \Nat$;
\item[(iii)] $f_n(A) \sub f(A)\cup\{0\}$ for every $n\in \Nat$.
\end{enumerate}
\end{fact}

\begin{lem}\label{lem:MeasureAlgebra}
Let $(\Delta,\Theta,\nu)$ be a probability space whose measure algebra is isomorphic to that of $(\Omega,\Sigma,\mu)$.
Then there is an isometry $T:L^1(\mu,X)\to L^1(\nu,X)$ such that both $T$ and $T^{-1}$ preserve $\delta\cS$-sets.
In particular, $X$ has property ($\delta\cS_\mu$) if and only if it has property ($\delta\cS_{\nu}$).
\end{lem}
\begin{proof}
The isomorphism between the measure algebras induces in a natural way an isometry $T:L^1(\mu,X)\to L^1(\nu,X)$ satisfying: 
\begin{enumerate}
\item[(i)] a set $K \sub L^1(\mu,X)$ is uniformly integrable if and only if $T(K)$ is;
\item[(ii)] $\nu_{T(f)}=\mu_{f}$ for every simple $f\in L^1(\mu,X)$.
\end{enumerate}

We claim that  $T(K)$ is a $\delta\cS$-set of $L^1(\nu,X)$ whenever $K \sub L^1(\mu,X)$ is a $\delta\cS$-set. Indeed,
$T(K)$ is uniformly integrable by~(i). Fix $\delta>0$ and a weakly compact set $W \sub X$ such that $\sup_{f\in K}\mu_f(X\setminus W)\leq \delta$.
Of course, we can assume that $0\in W$.
Take any $f\in K$. We will prove that $\nu_{T(f)}(X\setminus W)\leq \delta$. 
By Fact~\ref{fact:approx}, there is a sequence $(f_n)$ of simple functions converging to~$f$ in the norm of~$L^1(\mu,X)$ such that
$\mu_{f_n}(X\setminus W)\leq \delta$ for all $n\in \Nat$. Then (ii) implies that
$\nu_{T(f_n)}(X\setminus W) \leq \delta$ for all $n\in \Nat$. On the other hand, we also have $\|T(f_n)-T(f)\|_{L^1(\nu,X)}\to 0$ and so, by passing to a subsequence, 
we can assume further that $(T(f_n))$ converges to~$T(f)$ $\nu$-a.e. Since $X\setminus W$ is open, the function
$g:X \to \mathbb{R}$ defined by	
$$
	g(x):=\begin{cases}
	1 & \text{if $x\in X\setminus W$} \\
	0 & \text{if $x\in W$}
	\end{cases}
$$ 
is lower semicontinuous and so
$$
	g\circ T(f) \leq \liminf_{n\to \infty} \, g\circ T(f_n) \quad \nu\mbox{-a.e.}
$$
By integrating and applying Fatou's lemma, we get
\begin{multline*}
	\nu_{T(f)}(X\setminus W)=\int_\Delta g\circ T(f) \, d\nu \leq \int_\Delta \liminf_{n\to \infty} \, g\circ T(f_n) \, d\nu \\ \leq 
	\liminf_{n\to \infty}\, \int_\Delta g\circ T(f_n) \, d\nu =
	\liminf_{n\to \infty} \, \nu_{T(f_n)}(X\setminus W)\leq \delta.
\end{multline*}
This proves that $T(K)$ is a $\delta\cS$-set, as claimed.

By a symmetric argument, we also have that $T^{-1}(K')$ is a $\delta\cS$-set of $L^1(\mu,X)$
whenever $K' \sub L^1(\nu,X)$ is a $\delta\cS$-set.
\end{proof}

A probability space/measure is said to be {\em separable} if its $L^1$-space is separable or, equivalently, its
measure algebra equipped with the Fr\'{e}chet-Nikod\'{y}m metric is separable. A 
well-known result of Carath\'{e}odory states that the measure algebra of any non-atomic separable probability space is
isomorphic to the measure algebra of $([0,1],\mathcal{L},\lambda)$ (see e.g. \cite[Theorem~15.3.4]{roy-J}).

\begin{cor}\label{cor:SeparableAndSmaller}
The following assertions hold:
\begin{enumerate}
\item[(i)] If $\mu$ is separable and $X$ has property~($\delta\cS_\lambda$), then $X$ has property~($\delta\cS_{\mu}$).
\item[(ii)] If $\mu$ is not purely atomic and $X$ has property~($\delta\cS_\mu$), then $X$ has property~($\delta\cS_{\lambda}$).
\end{enumerate}
\end{cor}
\begin{proof}
Note first that Lemma~\ref{lem:MeasureAlgebra} and Carath\'{e}odory's theorem above imply that 
{\em properties ($\delta\cS_{\mu}$) and ($\delta\cS_{\lambda}$) are equivalent whenever $\mu$ is non-atomic and separable.}

(i): It only remains to consider the case in which $\mu$ is neither non-atomic nor purely atomic. 
Write $\Omega=A_1 \cup A_2$, where $A_i \in \Sigma$ with $\mu(A_i)>0$, $A_1\cap A_2=\emptyset$, $\mu_{A_1}$ is non-atomic
and $\mu_{A_2}$ is purely atomic. The later implies automatically that $X$ has property ($\delta\cS_{\mu_{A_2}}$).  
Since $\mu_{A_1}$ is non-atomic and separable, $X$ has property ($\delta\cS_{\mu_{A_1}}$). It is now
clear that $X$ has property ($\delta\cS_{\mu}$) as well.

(ii): Fix $A\in \Sigma$ with $\mu(A)>0$ such that $\mu_A$ is non-atomic. There is a sub-$\sigma$-algebra $\Sigma_0 \sub \Sigma_A$ such that the restriction $\mu_A|_{\Sigma_0}$
is non-atomic and separable. Let $T: L^1(\mu_A|_{\Sigma_0},X) \to L^1(\mu,X)$ be the isomorphic embedding
that maps each $f\in L^1(\mu_A|_{\Sigma_0},X)$ to the function $T(f)\in L^1(\mu,X)$ defined by 
$T(f)(\omega):=f(\omega)$ if $\omega\in A$, and $T(f)(\omega):=0$ if $\omega\in \Omega \setminus A$. It is easy to check that
$K \sub L^1(\mu_A|_{\Sigma_0},X)$ is a $\delta\cS$-set whenever $T(K) \sub L^1(\mu,X)$ is a $\delta\cS$-set. 
So, $X$ has property ($\delta\cS_{\mu_A|_{\Sigma_0}}$), which is equivalent to  ($\delta\cS_{\lambda}$).
\end{proof}

We next focus on Banach spaces which are subspaces of SWCG spaces. It is
important to stress that subspaces of SWCG spaces need not be SWCG, even in the separable case, see \cite[Section~3]{mer-sta-2}.

The following well-known fact is due to Grothendieck (see e.g. \cite[p.~227, Lemma~2]{die-J})
and will be used several times in the sequel.

\begin{fact}\label{fact:Grothendieck}
A set $W \sub X$ is relatively weakly compact if and only if for every $\epsilon>0$
there is a relatively weakly compact set $W_\epsilon\sub X$ such that $W \sub W_\epsilon +\epsilon B_X$.
\end{fact}

\begin{lem}\label{lem:CharacterizationDeltaS}
A set $K \sub L^1(\mu,X)$ is a $\delta\cS$-set if and only if it is uniformly integrable and 
for every $\delta>0$ and every $\epsilon>0$ there exists a weakly compact set $W \sub X$ such that
$\mu(f^{-1}(W+\epsilon B_X))\geq 1-\delta$ for every $f\in K$.
\end{lem}
\begin{proof} We only have to prove the ``if'' part.
Fix $\delta>0$. Let $(W_n)$ be a sequence of weakly compact subsets of~$X$ such that 
$\mu(f^{-1}(W_n+2^{-n} B_X))\geq 1-\delta/2^{n}$ for every $f\in K$. 
Then 
$$
	W:=\bigcap_{n\in \Nat} \Big(W_n+\frac{1}{2^{n}}B_X\Big)
$$ 
is weakly compact (Fact~\ref{fact:Grothendieck})
and for each $f\in K$ we have 
$$
	\mu\big(\Omega \setminus f^{-1}(W)\big)
	\leq
	\sum_{n\in \Nat}
	\mu\Big(\Omega \setminus f^{-1}\Big(W_n+\frac{1}{2^{n}}B_X\Big)\Big) \leq \delta.
$$
This shows that $K$ is a $\delta\cS$-set.
\end{proof}

\begin{pro}\label{pro:SWCGcountable}
Suppose $X$ is a subspace of a SWCG space. Let $K \sub L^1(\mu,X)$. If every countable
subset of~$K$ is a $\delta\cS$-set, then $K$ is a $\delta\cS$-set.
\end{pro}
\begin{proof}
Clearly, we can assume without loss of generality that $X$ is SWCG. Let $G \sub X$ be a weakly compact set which strongly generates 
the family of all weakly compact subsets of~$X$.
Since every countable subset of~$K$ is uniformly integrable, so is~$K$. By contradiction,
suppose that $K$ is not a $\delta\cS$-set. According to Lemma~\ref{lem:CharacterizationDeltaS},
there exist $\delta>0$, $\epsilon>0$ and a sequence $(f_n)$ in~$K$ such that 
$$
	\mu\big(f_n^{-1}(nG+\epsilon B_X)\big) <1-\delta
	\quad
	\mbox{for all }n\in \Nat.
$$
On the other hand, since every countable subset of~$K$ is a $\delta\cS$-set, there is a weakly compact set $W \sub X$ such that
$\mu(f_n^{-1}(W))\geq 1-\delta$ for all $n\in \Nat$. It follows that $W \not \subseteq nG+\epsilon B_X$
for all $n\in \Nat$, which contradicts the choice of~$G$. 
\end{proof}

\begin{cor}\label{cor:SWCGcountable}
Suppose that $X$ is a subspace of a SWCG space and that $\mu$ is not purely atomic. Then $X$ has property ($\delta\cS_\mu$) if and only if it has property ($\delta\cS_{\lambda}$).
\end{cor}
\begin{proof}
In view of Corollary~\ref{cor:SeparableAndSmaller}(ii) it only remains to prove the ``if'' part.  
By Proposition~\ref{pro:SWCGcountable}, in order to show that $X$ has property~($\delta\cS_\mu$)
it suffices to check that every {\em countable} relatively weakly compact set $K\sub L^1(\mu,X)$
is a $\delta\cS$-set. The fact that $K$ is countable ensures the 
existence of a sub-$\sigma$-algebra
$\Sigma_0 \sub \Sigma$ such that the restriction $\mu|_{\Sigma_0}$ is separable
and $K \sub L^1(\mu|_{\Sigma_0},X)$ (identified as a subspace of~$L^1(\mu,X)$).
Since $X$ has property~($\delta\cS_{\mu|_{\Sigma_0}}$) (by Corollary~\ref{cor:SeparableAndSmaller}(i)), $K$ is a $\delta\cS$-set 
of~$L^1(\mu|_{\Sigma_0},X)$, and so of~$L^1(\mu,X)$. 
\end{proof}

For any weakly compact set $W \sub X$ we have an obvious $\delta\mathcal{S}$-set, namely, 
\begin{equation}\label{eqn:LC}
    L(W):=\{f\in L^1(\mu,X): \, f(\omega)\in W \mbox{ for }\mu\mbox{-a.e. }\omega\in \Omega\}.
\end{equation}  
We next consider a generalization of such sets.
Let $wk(X)$ be the family of all weakly compact non-empty subsets of~$X$.
Given a multi-function $F:\Omega \to wk(X)$, we define
$$
	S^1(F):=\{f\in L^1(\mu,X): \, f(\omega)\in F(\omega) \mbox{ for }\mu\mbox{-a.e. }\omega\in \Omega\}.
$$
Any uniformly integrable subset of $S^1(F)$ is relatively weakly compact in $L^1(\mu,X)$,
see \cite[Corollary~2.6]{die-alt2}. The following result is implicit in the proof of \cite[Lemma~2]{laj-rod} --
we include a detailed proof for the reader's convenience. 

\begin{pro}\label{pro:LR}
Suppose $X$ is a subspace of a SWCG space. Let $F:\Omega \to wk(X)$ be a multi-function such that
\begin{equation}\label{eqn:meas}
	\{\omega \in \Omega: \ F(\omega)\sub C\}\in \Sigma
	\quad
	\mbox{for every convex closed set }C \sub X.
\end{equation}
Then: 
\begin{enumerate}
\item[(i)] For every $\delta>0$ there exist $A\in \Sigma$ with $\mu(A)\geq 1-\delta$
and a weakly compact set $W \sub X$ such that $F(\omega)\sub W$ for every $\omega\in A$.
\item[(ii)] Every uniformly integrable subset of $S^1(F)$ is a $\delta\cS$-set.
\end{enumerate}
\end{pro}
\begin{proof} (ii) follows immediately from (i). To prove~(i) 
we can suppose without loss of generality that $X$ is SWCG, i.e. $wk(X)$ is strongly generated by some 
weakly compact set $G \sub X$. By the Krein-Smulyan theorem we can assume further that $G$ is absolutely convex.  
For each $n,k\in \Nat$, the set $C_{n,k}:=nG + (1/k) B_X$ is convex and closed,
hence~\eqref{eqn:meas} ensures that
$$
	A_{n,k}:=\{\omega \in \Omega: \, F(\omega) \sub C_{n,k}\} \in \Sigma.
$$
For each $k\in \Nat$ we have $\Omega=\bigcup_{n\in \Nat}A_{n,k}$, the union being increasing  
since $G$ is balanced. Fix $\delta>0$. For each $k\in \Nat$ we can choose $n_k\in \Nat$ large enough such that
$\mu(\Omega \setminus A_{n_k,k})\leq \delta/2^{k}$. Define $A:=\bigcap_{k\in \Nat}A_{n_k,k}\in \Sigma$, so
that $\mu(\Omega \setminus A) \leq \delta$. Now, the set $W:=\bigcap_{k\in \Nat} C_{n_k,k}$
is weakly compact (by Fact~\ref{fact:Grothendieck}) and $F(\omega)\sub W$ for every $\omega\in A$.
This finishes the proof.
\end{proof}

The examples of Subsections~\ref{subsection:Talagrand} and~\ref{subsection:BH} below
are built over Banach spaces constructed using the {\em dyadic tree}
$\cT:=\bigcup_{n\in \Nat}\{0,1\}^n$, that is, the set of all (non-empty) finite sequences of $0$'s and $1$'s.
$\cT$ is equipped with the partial order $\succeq$ defined by declaring that 
$\tau\succeq \sigma$ if $\tau$ extends~$\sigma$. 
By a {\em chain} (resp. {\em antichain}) of~$\cT$ we mean a set of pairwise comparable (resp. incomparable)
elements of~$\cT$. We denote by $c_{00}(\cT)$ the linear space of all finitely supported real-valued functions on~$\cT$.
For each $\sigma\in \cT$, let $e_\sigma\in c_{00}(\cT)$ be defined by $e_\sigma(\tau):=0$
for all $\tau\neq \sigma$, and $e_\sigma(\sigma):=1$.

\subsection{The Talagrand space}\label{subsection:Talagrand}

The Banach space $E$ is defined as the completion of $c_{00}(\cT)$ equipped with the norm
$$
	\|x\|_{E}:=\sup_{n\in \Nat}
	\left(
	\sum_{\sigma\in \{0,1\}^n} \left(\sup_{\tau\succeq \sigma}|x(\tau)|\right)^2
	\right)^{1/2}.
$$
This space was introduced by Talagrand in~\cite{tal15} and served to refute
conjectures about the Dunford-Pettis, Banach-Saks and Kadec-Klee properties
in Lebesgue-Bochner spaces (cf. \cite[Section~5.5]{lin-J}). The set $\{e_\sigma:\sigma\in \cT\}$ is 
an unconditional basis of~$E$.
The dual $E^*$ is separable and has the Schur property, see \cite[Corollary~2]{tal15} 
(cf. \cite[Theorem~5.5.3]{lin-J}), hence it is SWCG. 

\begin{theo}\label{theo:Talagrand}
$E^*$ fails property~($\delta\mathcal{S}_\lambda$).
\end{theo}
\begin{proof} 
We work with $\Omega:=\prod_{n\in \Nat}\{0,1\}^n$ and the usual probability measure~$\mu$ on~$\Sigma:={\rm Borel}(\Omega)$, 
that is, the product of the probability measures assigning measure $2^{-n}$ to each singleton of~$\{0,1\}^n$. 
Since $(\Omega,\Sigma,\mu)$ is non-atomic and separable, Corollary~\ref{cor:SeparableAndSmaller}
makes clear that it suffices to show that $E^*$ fails property~($\delta\mathcal{S}_\mu$). 

Let $\{e_\sigma^*:\sigma \in \cT\} \sub E^*$
be the set of biorthogonal functionals associated to the basis $\{e_\sigma:\sigma\in \cT\}$.
For each $n\in \Nat$, let $P_n:\Omega \to \{0,1\}^n$ be the $n$th-coordinate projection and
define a simple function $f_n:\Omega \to E^*$ by 
$$
	f_n(\omega):=e^*_{P_n(\omega)}.
$$  
The sequence $(f_n)$ is weakly null in $L^1(\mu,E^*)$, see \cite{tal15}, proof of Theorem~3 (cf. \cite[Theorem~5.5.4]{lin-J}).
We claim that the relatively weakly compact set 
$$
	K:=\{f_n: \, n\in \Nat\} \sub L^1(\mu,E^*)
$$ 
is not a $\delta\cS$-set. Actually, we will prove that {\em for every} $\delta>0$ and every weakly compact set $W \sub E^*$
there is $n\in \Nat$ such that $\mu(f_n^{-1}(W))< \delta$.

By contradiction, suppose that there exist $\delta>0$ and a weakly compact set $W \sub E^*$ 
such that $\mu(f_n^{-1}(W)) \geq \delta$ for all $n\in \Nat$. Since $E^*$ has the Schur property, $W$ 
is norm compact. Note that
$$
	\mu\Big(\bigcap_{n\in \Nat}\bigcup_{m\geq n}f_m^{-1}(W)\Big) \geq \delta.
$$
In particular, there exist $\omega\in \Omega$ and a subsequence $n_1<n_2<\dots$ of~$\Nat$
such that $f_{n_k}(\omega)=e^*_{P_{n_k}(\omega)}\in W$ for all $k\in \Nat$. 
Since $P_{n_k}(\omega)\neq P_{n_{k'}}(\omega)$ whenever $k\neq k'$ and $\{e_\sigma:\sigma\in \cT\}$ is 
a normalized unconditional basis of~$E$, the sequence $(e^*_{P_{n_k}(\omega)})$ is $w^*$-null in~$E^*$.
Bearing in mind that $\|e^*_\sigma\|_{E^*}=1$ for all $\sigma\in \cT$, it follows 
that $(e^*_{P_{n_k}(\omega)})$ cannot have norm convergent subsequences, which
contradicts the norm compactness of~$W$. This finishes the proof.
\end{proof}

\subsection{The Batt-Hiermeyer space}\label{subsection:BH}

Let $X_0:=\{x: \cT\to \mathbb{R}: \, \|x\|_{X_0}<\infty\}$, where
$$
	\|x\|_{X_0}:=\sup
	\left\{\left(
	\sum_{i=1}^n 
	\left(
	\sum_{\sigma\in C_i}|x(\sigma)|
	\right)^2
	\right)^{1/2}
	\right\},
$$
the supremum being taken over all finite sets $\{C_1,\dots,C_n\}$ of pairwise
disjoint finite chains of~$\cT$. The Banach space $(X_0,\|\cdot\|_{X_0})$ was 
introduced in~\cite{bat-hie}. A more general construction
dealing with the so-called ``adequate families'' was considered by Kutzarova and Troyanski~\cite{kut-tro-2} (cf. \cite[Section~3]{arg-mer}). 

The set $\{e_\sigma:\sigma\in \cT\}$ is a boundedly complete 
unconditional basis of~$X_{0}$. In particular, $X_0$ is separable and WSC. 
As we mentioned in the introduction, it is known that $X_0$ cannot be embedded into a SWCG space. We next
give another proof of this fact. To this end, we use Proposition~\ref{pro:LR} and the example of~\cite[Section~3]{bat-hie} 
showing that $X_0$ fails property~($\delta\cS_\lambda$).

\begin{theo}\label{exa:BH}
$X_{0}$ is not a subspace of a SWCG space.
\end{theo}
\begin{proof}
We work with $\Omega:=\{0,1\}^\Nat$ equipped with the standard probability~$\mu$ on $\Sigma:={\rm Borel}(\Omega)$. For each $n\in \Nat$,
we define $\Delta_n: \Omega \to \{0,1\}^n$ by 
$$
	\Delta_n(\omega):=(\omega_1,\dots,\omega_{n-1},1-\omega_n), \quad \omega=(\omega_k)\in \Omega,
$$
and we define a simple function $f_n:\Omega \to X_{0}$ by $f_n(\omega):=e_{\Delta_n(\omega)}$. Observe that for each $\omega\in \Omega$
the set $\{\Delta_n(\omega):n\in \Nat\}$ is an infinite antichain of~$\cT$ and so the sequence $(e_{\Delta_n(\omega)})$ is equivalent to the
usual basis of~$\ell^2$. In particular, the sequence $(f_n(\omega))$ is weakly null in~$X_{0}$ for every $\omega\in \Omega$. Bearing in mind that
$(f_n)$ is uniformly bounded, it is easy to check that $(f_n)$ is weakly null in $L^1(\mu,X_{0})$. 

In \cite[Section~3]{bat-hie} it was proved that $K:=\{f_n:n\in \Nat\}$ is not a $\delta\cS$-set of $L^1(\mu,X_{0})$. On the other hand, the multi-function
$$
	F:\Omega \to wk(X_{0}), \quad
	F(\omega):=\{f_n(\omega):n\in \Nat\}\cup\{0\},
$$  
satisfies condition~\eqref{eqn:meas} of Proposition~\ref{pro:LR}. Indeed, for every weakly closed set $C \sub X_{0}$ we have
$$
	\{\omega \in \Omega: \ F(\omega)\sub C\}=\Omega \setminus \bigcup_{n\in \Nat}f_n^{-1}(X_{0}\setminus C) \in \Sigma,
$$
because each $f_n$ is a simple function. Note that $K$ is a uniformly integrable subset of~$S^1(F)$
which is not a $\delta\cS$-set. From Proposition~\ref{pro:LR}(ii) it follows that $X$ cannot be a subspace of a SWCG space.
\end{proof}

\section{Further results}\label{section:results}

We begin this section with a Grothendieck-type stability result for $\delta\cS$-sets
which involves the sets $L(W)$ defined in~\eqref{eqn:LC}.

\begin{pro}\label{pro:GrothendieckStable}
Let $K \sub L^1(\mu,X)$. The following statements are equivalent:
\begin{enumerate}
\item[(i)] $K$ is a $\delta\mathcal{S}$-set.
\item[(ii)] For every $\epsilon>0$ there is a weakly compact set $W \sub X$ such that
$$
	K \sub L(W)+\epsilon B_{L^1(\mu,X)}.
$$
\item[(iii)] For every $\epsilon>0$ there is a $\delta\mathcal{S}$-set $K_\epsilon \sub L^1(\mu,X)$ such that
$$
	K \sub K_{\epsilon}+\epsilon B_{L^1(\mu,X)}.
$$
\end{enumerate}
\end{pro}
\begin{proof}
(i)$\impli$(ii): See \cite[Fact~2.5]{rod13} for an easy proof. 

(ii)$\impli$(iii): Obvious.

(iii)$\impli$(i): Clearly, $K$ is uniformly integrable. Fix $\delta>0$. Bearing in mind (i)$\impli$(ii), we can choose a sequence $(W_n)$ of weakly
compact subsets of~$X$ such that
\begin{equation}\label{eqn:Grothendieck}
	K \sub L(W_n) + \frac{\delta}{4^n}B_{L^1(\mu,X)} \quad
	\mbox{for every }n\in \Nat.
\end{equation}
By Fact~\ref{fact:Grothendieck}, the set
$$
	W:=\bigcap_{n\in \Nat} \Big(W_n+\frac{1}{2^n}B_X\Big)
$$
is weakly compact. {\em We claim that $\mu(f^{-1}(W))\geq 1-\delta$ for every $f\in K$.} Indeed,
for each $n\in \Nat$ we use \eqref{eqn:Grothendieck} to find $h_n\in L^1(\mu,X)$ such that $f-h_n\in L(W_n)$ and 
$\|h_n\|_{L^1(\mu,X)}\leq \delta/4^{n}$. Chebyshev's inequality applied to~$h_n$ yields
$$
	\mu\Big(\Omega\setminus f^{-1}\Big(W_n+\frac{1}{2^n}B_X\Big)\Big) \leq 
	\mu\Big(\Big\{\omega\in \Omega:\, \|h_n(\omega)\|_X>\frac{1}{2^n}\Big\}\Big)\leq \frac{\delta}{2^n}.
$$ 
Therefore
$$
	\mu\big(\Omega\setminus f^{-1}(W)\big)\leq \sum_{n\in \Nat} \mu\Big(\Omega\setminus f^{-1}\Big(W_n+\frac{1}{2^n}B_X\Big)\Big) \leq \delta.
$$
This proves the claim. It follows that $K$ is a $\delta\mathcal{S}$-set.
\end{proof}

We say that a set $H \sub L^1(\mu,X)$ is {\em uniformly bounded} if there is a constant $C>0$
such that for every $f\in H$ we have $\|f(\cdot)\|_X \leq C$ $\mu$-a.e. Recall that
a subset of a Banach space is said to be {\em conditionally weakly compact} if every sequence in it
admits a weakly Cauchy subsequence.

\begin{theo}\label{theo:UniformlyBounded}
$X$ has property~($\delta\mathcal{S}_\mu$) if and only if every uniformly bounded 
relatively weakly compact subset of~$L^1(\mu,X)$ is a $\delta\mathcal{S}$-set.
\end{theo}
\begin{proof} Suppose that $\mu$ is not purely atomic (otherwise any Banach space has property~($\delta\mathcal{S}_\mu$))
and that every uniformly bounded 
relatively weakly compact subset of~$L^1(\mu,X)$ is a $\delta\mathcal{S}$-set. We begin by proving the following: 

{\em Claim. $X$ is WSC.} 
The argument is the same that property~($\delta\mathcal{S}_\lambda$) implies weak sequential completeness 
(cf. \cite[Section~4]{bat-hie}), but we include it for the reader's convenience. By arguing as in the proof 
of Corollary~\ref{cor:SeparableAndSmaller}(ii), we can 
assume without loss of generality that $\mu=\lambda$. Let $(r_n)$
be the sequence of Rademacher functions on~$[0,1]$.
Given any weakly Cauchy sequence $(x_n)$ in~$X$, the sequence $(r_n x_n)$ is weakly null in~$L^1([0,1],X)$
(see \cite[Proposition~4.1]{bat-hie}). The set $\{r_n x_n:n\in\Nat\}$ is uniformly bounded
and relatively weakly compact, hence it is a $\delta\mathcal{S}$-set. Let $W \sub X$ be a weakly compact set
such that $\lambda((r_n x_n)^{-1}(W))>0$ for all $n\in \Nat$. 
Then the (weakly Cauchy) sequence $(x_n)$
is contained in the weakly compact set $W\cup(-W)$ and so it is weakly convergent. This proves that $X$ is WSC, as claimed.

Let $K \sub L^1(\mu,X)$ be a relatively weakly compact set. In particular, $K$ is conditionally weakly compact
and so the same holds for $\{f\chi_A:f\in K,A\in \Sigma\}$, see~\cite[Proposition~10]{bou-JJ}
(cf. \cite[Corollary~4.2]{bat-hie}).
A standard argument (cf. \cite[Theorem~2.18]{ghe}) now ensures
that for every $\epsilon>0$ there is a {\em uniformly bounded} conditionally weakly compact set $K_\epsilon \sub L^1(\mu,X)$
such that 
$$
	K \sub K_\epsilon+ \epsilon B_{L^1(\mu,X)}.
$$ 
On the other hand, since $X$ is WSC, the same holds for $L^1(\mu,X)$, thanks
to a deep result of Talagrand,~\cite[Theorem~11]{tal11} (cf. \cite[Theorem~5.3.14]{lin-J}). Hence the $K_\epsilon$'s above are relatively weakly compact and 
so they are $\delta\mathcal{S}$-sets. An appeal to Proposition~\ref{pro:GrothendieckStable}
ensures that $K$ is a $\delta\mathcal{S}$-set.
\end{proof}

The following result provides new examples of Banach spaces having property~($\delta\mathcal{S}_\mu$).
For instance, it shows that the space $\ell^r(\ell^1)$ has property~($\delta\mathcal{S}_\mu$) for any $1<r<\infty$. 
This space is not a subspace of a SWCG space, see \cite[Corollary~2.29]{kam-mer2} (cf. \cite[Corollary~2.9]{laj-rod-2}).

Recall that the $\ell^r$-sum ($1\leq r <\infty$) of a sequence of Banach spaces~$(X_n)$ is the Banach space
$\ell^r(X_n)$ consisting of all sequences $(x_n)\in \prod_{n\in \Nat}X_n$ such that
$$
	\big\|(x_n)\big\|_{\ell^r(X_n)}=\Big(\sum_{n\in \Nat}\|x_n\|_{X_n}^r\Big)^{1/r}<\infty.
$$

\begin{theo}\label{theo:StabilitySums}
Let $(X_n)$ be a sequence of Banach spaces having property~($\delta\mathcal{S}_\mu$).
Then $\ell^r(X_n)$ has property~($\delta\mathcal{S}_\mu$) for every $1\leq r <\infty$.
\end{theo}
\begin{proof}
Write $X:=\ell^r(X_n)$ and, for each $n\in \Nat$, let $\pi_n:X \to X_n$ be the $n$th-coordinate projection
and $\tilde{\pi}_n:L^1(\mu,X)\to L^1(\mu,X_n)$ the operator defined by $\tilde{\pi}_n(f):=\pi_n\circ f$
for all $f\in L^1(\mu,X)$. 

Fix an arbitrary relatively weakly compact set $K \sub L^1(\mu,X)$. For
each $n\in \Nat$, the set $\tilde{\pi}_n(K)=\{\pi_n\circ f:f\in K\}$ is relatively
weakly compact (hence a $\delta\mathcal{S}$-set) in~$L^1(\mu,X_n)$,
and so for every $\delta>0$ there is a weakly compact set $W^\delta_n\sub X_n$ such that
\begin{equation}\label{eqn:Projection}
	\mu\big((\pi_n\circ f)^{-1}(W^\delta_n)\big) \geq 1-\frac{\delta}{2^n}
	\quad
	\mbox{for every }f\in K.
\end{equation}
We now divide the proof that $K$ is a $\delta\cS$-set into two cases.

{\em Case $r=1$.} It is easy to prove that the map
$$
	\Phi: L^1(\mu,X) \to \ell^1\big(L^1(\mu,X_n)\big),
	\quad
	\Phi(f):=(\pi_n\circ f),
$$
is an isometric isomorphism.  We will show that $K$ is a $\delta\mathcal{S}$-set by checking
condition~(iii) in Proposition~\ref{pro:GrothendieckStable}. Fix $\epsilon>0$. Since $\Phi(K)$
is relatively weakly compact, there is $N\in \Nat$ such that
\begin{equation}\label{eqn:Truncation}
	\sup_{f\in K}\, \sum_{n>N}\|\pi_n\circ f\|_{L^1(\mu,X_n)}\leq \epsilon
\end{equation}
(see e.g. \cite[Lemma~7.2(ii)]{kac-alt}). For each $f\in L^1(\mu,X)$, we define $P(f)\in L^1(\mu,X)$ by declaring
$$
	\pi_n\circ P(f):=\begin{cases}
	\pi_n\circ f & \text{if $n\leq N$}, \\
	0 & \text{if $n>N$}.
	\end{cases}
$$
Since $P:L^1(\mu,X)\to L^1(\mu,X)$ is an operator, 
$P(K)$ is relatively weakly compact in~$L^1(\mu,X)$. Observe also that \eqref{eqn:Truncation} yields 
$$
	K \sub P(K)+\epsilon B_{L^1(\mu,X)}. 
$$
We claim that $P(K)$ is a $\delta\mathcal{S}$-set. Indeed, fix $\delta>0$ and 
consider the weakly compact subset of~$X$ defined by
$$
	W:=\bigcap_{n=1}^N\pi_n^{-1}(W_n^\delta)\cap \bigcap_{n>N}\pi_n^{-1}(\{0\}).
$$
Then for each $f\in K$ we have
$$
	\mu\big(\Omega \setminus P(f)^{-1}(W)\big)\leq 
	\sum_{n=1}^N\mu\big(\Omega \setminus (\pi_n\circ f)^{-1}(W^\delta_n)\big) 
	\stackrel{\eqref{eqn:Projection}}{\leq} \delta.
$$ 
This shows that $P(K)$ is a $\delta\mathcal{S}$-set. The proof of the case $r=1$ is finished.

{\em Case $1<r<\infty$.} By Theorem~\ref{theo:UniformlyBounded}, we can assume that $K$ is uniformly bounded.
Let $C>0$ be a constant such that for every $f\in K$ we have $\|f(\cdot)\|_X \leq C$ $\mu$-a.e. 
Fix $\delta>0$. Then the set 
$$
	W:=\bigcap_{n\in \Nat}\pi_{n}^{-1}(W^\delta_n) \cap C B_X
$$
is weakly compact in~$X$ (here we use that $1<r<\infty$) and satisfies
$$
	\mu\big(\Omega \setminus f^{-1}(W)\big)
		\leq \sum_{n\in \Nat}\mu\big(\Omega\setminus (\pi_n\circ f)^{-1}(W^\delta_n)\big)\stackrel{\eqref{eqn:Projection}}{\leq} \delta
	\quad\mbox{for every }f\in K.
$$
This shows that $K$ is a $\delta\mathcal{S}$-set and finishes the proof of the theorem. 
\end{proof}

\begin{rem}\label{rem:WCGrange}
In general, it is not true that for every relatively weakly compact set $K \sub L^1(\mu,X)$ there
a weakly compactly generated (WCG) subspace $Y \sub X$ such that $\mu(f^{-1}(Y))=1$
for every $f\in K$, see \cite[Remark on p.~434]{tal1} and~\cite{rod14}. However, this is the case if $K$ is a $\delta\cS$-set. Indeed,
if $(W_n)$ is a sequence of weakly compact subsets of~$X$ such that 
$\mu(f^{-1}(W_n))\geq 1-1/n$ for every $f\in K$ and every $n\in \Nat$, then
$$
	Y:=\overline{{\rm span}}\Big(\bigcup_{n\in \Nat}W_n\Big)
$$
is WCG and $\mu(f^{-1}(Y))=1$ for every $f\in K$.
\end{rem}

We finish the paper by discussing the concept of~$\delta\cS$-set in the Lebesgue-Bochner
space $L^p(\mu,X)$ for $1<p<\infty$. Recall
that $L^p(\mu,X)$ is the Banach space of all (equivalence classes of) strongly measurable functions $f:\Omega\to X$ such that 
$$
	\|f\|_{L^p(\mu,X)}=\left(\int_\Omega \|f(\cdot)\|^p_X \, d\mu\right)^{1/p} < \infty.
$$
We denote by 
$$
	j_p:L^p(\mu,X)\to L^1(\mu,X)
$$ 
the inclusion operator. The following result collects two well-known properties of~$j_p$. 

\begin{fact}\label{fact:Moving}
Let $H \sub L^p(\mu,X)$ be a bounded set. Then:
\begin{enumerate}
\item[(i)] $j_p(H)$ is uniformly integrable.
\item[(ii)] $H$ is relatively weakly compact if and only if $j_p(H)$ is relatively weakly compact.
\end{enumerate}
\end{fact}

Statement~(i) is a simple consequence of H{\"o}lder's inequality. The non-trivial implication of~(ii) can be proved easily by using the standard representation of $L^1(\mu,X)^*$
and $L^p(\mu,X)^*$ as spaces of vector-valued functions.
It also follows immediately from the characterizations of relative weak compactness in~$L^1(\mu,X)$
and $L^p(\mu,X)$ given in~\cite{die-alt2}. In operator theoretic terms, (ii) means that $j_p$ is {\em tauberian}
(see \cite[Section~3.4]{gon-abe}).  

A set $K \sub L^p(\mu,X)$ is said to be a {\em $\delta\mathcal{S}$-set} (see \cite[Section~3]{die5}) if it is bounded and
for every $\delta>0$ there is a weakly compact set $W \sub X$ such that
$\mu(f^{-1}(W))\geq 1-\delta$ for every $f\in K$. By Fact~\ref{fact:Moving}(i), this is equivalent to saying
that $K$ is bounded and $j_p(K)$ is a $\delta\cS$-set of~$L^1(\mu,X)$.
From Fact~\ref{fact:Moving}(ii) and Theorem~\ref{theo:Diestel} it follows at once that
{\em every $\delta\cS$-set of~$L^p(\mu,X)$ is relatively weakly compact}.

\begin{cor}\label{cor:DeltaS-p}
$X$ has property~($\delta\cS_\mu$) if and only if every relatively weakly compact subset of~$L^p(\mu,X)$ is a $\delta\cS$-set. 
\end{cor}
\begin{proof}
The ``only if'' part is immediate. Conversely, suppose that every relatively weakly compact subset of~$L^p(\mu,X)$ is a $\delta\cS$-set.
We next check that $X$ has property~($\delta\cS_\mu$) by testing on uniformly bounded sets (Theorem~\ref{theo:UniformlyBounded}).
Let $K$ be a uniformly bounded relatively weakly compact subset of~$L^1(\mu,X)$. Since $K$ is uniformly bounded, it can be 
identified with a bounded subset $\tilde{K}$ of~$L^p(\mu,X)$, i.e. $j_p(\tilde{K})=K$. 
By Fact~\ref{fact:Moving}(ii), $\tilde{K}$ is relatively weakly compact in $L^p(\mu,X)$
and so it is a $\delta\cS$-set. Hence $K$ is a $\delta\cS$-set of $L^1(\mu,X)$. 
\end{proof}

It is worth pointing out that, in general, the implication (i)$\impli$(ii) in Proposition~\ref{pro:GrothendieckStable}
does not work for~$L^p(\mu,X)$. More precisely:

\begin{rem}\label{rem:NoGrothendieck}
For each weakly compact set $W \sub X$, let
$$
	\tilde{L}(W):=\{f\in L^p(\mu,X): \, f(\omega)\in W \mbox{ for }\mu\mbox{-a.e. }\omega\in \Omega\}.
$$
Suppose that the following condition holds:
\begin{enumerate}
\item[(*)] For every $\delta\mathcal{S}$-set $K \sub L^p(\mu,X)$ and every $\epsilon>0$ there is a weakly compact set $W \sub X$ such that
$K \sub \tilde{L}(W)+\epsilon B_{L^p(\mu,X)}$.
\end{enumerate}
Then we have:
\begin{enumerate}
\item[(i)] If $X$ is SWCG, then the family of all $\delta\cS$-sets of $L^p(\mu,X)$ is strongly generated 
(imitate the argument in \cite[proof of Theorem~2.7]{rod13}). 
\item[(ii)] Therefore, if $X$ is SWCG and has property~($\delta\cS_\mu$), then $L^p(\mu,X)$ is SWCG. 
\end{enumerate}
Suppose in addition that $L^1(\mu)$ is infinite-dimensional. In this case,  
$L^p(\mu,X)$ is SWCG if and only if $X$ is reflexive (see \cite[Corollary~3.7]{rod13}). 
Summing up, we conclude that condition~(*) fails for spaces like $X=L^1[0,1]$ or $X=\ell^1$.
\end{rem}

\subsection*{Acknowledgements}
Research supported by Ministerio de Econom\'{i}a y Competitividad and 
FEDER (project MTM2014-54182-P).
This work was also supported by the research project 19275/PI/14 funded by Fundaci\'{o}n S\'{e}neca - Agencia de Ciencia y Tecnolog\'{i}a 
de la Regi\'{o}n de Murcia within the framework of PCTIRM 2011-2014.


\def\cprime{$'$}\def\cdprime{$''$}
  \def\polhk#1{\setbox0=\hbox{#1}{\ooalign{\hidewidth
  \lower1.5ex\hbox{`}\hidewidth\crcr\unhbox0}}} \def\cprime{$'$}
\providecommand{\bysame}{\leavevmode\hbox to3em{\hrulefill}\thinspace}
\providecommand{\MR}{\relax\ifhmode\unskip\space\fi MR }
\providecommand{\MRhref}[2]{%
  \href{http://www.ams.org/mathscinet-getitem?mr=#1}{#2}
}
\providecommand{\href}[2]{#2}


\bibliographystyle{amsplain}

\end{document}